\documentclass[11pt]{amsart}
\usepackage[utf8]{inputenc}

\usepackage{amssymb,latexsym,graphics,enumerate}
\usepackage[normalem]{ulem}
\usepackage[mathscr]{eucal}
\usepackage{fullpage}
\usepackage{color}
\usepackage{url}
\usepackage{amsmath}
\usepackage{graphicx}
\usepackage{multirow}

\theoremstyle{plain}
        \newtheorem{theorem}{Theorem}

        \newtheorem{lemma}[theorem]{Lemma}

\numberwithin{equation}{section}
\let\oldmarginpar\marginpar
\renewcommand\marginpar[1]{\-\oldmarginpar[\raggedleft\footnotesize #1]%
{\raggedright\footnotesize #1}}

\begin{document}

\title{A Sufficient Condition for Blowup of the Nonlinear Klein-Gordon Equation with Positive Initial Energy in FLRW Spacetimes}
\author{}

%\thanks{Department of Mathematics, California %State University, East Bay, 
%25800 Carlos Bee Boulevard, Hayward, CA 94542, %USA. 
%\email{jmccollum@horizon.csueastbay.edu }
%\email{gmwamba@horizon.csueastbay.edu }
%\email{jesus.oliver@csueastbay.edu}
%}

\date{}

\maketitle

\begin{center}

\bigskip
Jonathon McCollum \footnote{email address: mccolluj@oregonstate.edu}{${}^{,\star}$},
Gregory Mwamba \footnote{email address: gmwamba@ucmerced.edu}{${}^{,\#}$}, and 
Jes\'{u}s Oliver \footnote{email address: jesus.oliver@csueastbay.edu}{${}^{,\dagger}$}

\bigskip
{\it {${}^\star$Oregon State University,}}\\
{\it{Department of Mathematics,}}\\
{\it {Kidder Hall 368,}}\\
{\it {Corvallis, Oregon, 
USA, 97331-4605}}

\bigskip
{\it {${}^\#$}University of California, Merced,}\\
{\it{Department of Applied Mathematics,}}\\
{\it {5200 North Lake Road,}}\\
{\it {Merced, California, USA, 95343}}

\bigskip
{\it {${}^\dagger$}California State University East Bay,}\\
{\it{Department of Mathematics,}}\\
{\it {25800 Carlos Bee Boulevard,}}\\
{\it {Hayward, California, USA, 94542}}

\begin{abstract}
In this paper we demonstrate a sufficient condition for blowup of the
nonlinear Klein-Gordon equation with arbitrarily positive initial energy in
Friedmann-Lemaître-Robertson-Walker spacetimes. This is accomplished using an 
established concavity method that has been employed for similar PDEs in Minkowski
space. This proof relies on the energy inequality associated with this equation,
$E(t_0)\geq E(t)$, also proved herein using a geometric method.
\end{abstract}

\end{center}

\section{Introduction}

The nonlinear Klein-Gordon equations
are a class of important evolution equations 
that describe the movement of spinless
relativistic particles, which can lend understanding in many 
physical applications (see \cite{Yangxu}
and references therein). In this paper we consider the question of finite time blow up for a 
nonlinear class of Klein-Gordon equations propagating on a fixed cosmological 
spacetime. We have adapted the methods of Yang \& Xu in \cite{Yangxu} and Wang 
in \cite{Wang}, both of which prove blowup results for solutions to the Klein-Gordon 
equation in a non-expanding spacetime. The result from \cite{Wang} 
was the first to show a condition for blowup using this general method, and this work 
was expanded upon in \cite{Yangxu} to broaden the conditions that reliably lead to 
blowup. Both of these were done under the assumption of a Minkowski 
background metric -- for which energy conservation for Klein-Gordon is readily available. On the other hand,
our work focuses on adapting this method of proof to backgrounds
that include expanding spacetimes -- for which only an energy
inequality is available. Due
to the accelerating nature of the expansion of the
observable universe, such cosmological spacetimes are of
general interest. The results in this paper also supplement
those in \cite{Gal1}. In that work, the authors
establish small data global existence for 
a class of semilinear Klein-Gordon equations propagating
in a family of expanding cosmological
backgrounds which include the de Sitter spacetime.
On the other hand,
in this paper we are considering {\bf{large}} initial data
satisfying a specific set of conditions that ensure
$L^2$ blow up in finite time.

The general concavity method employed to prove
the blowup result in this paper, along with
in \cite{Wang} and \cite{Yangxu}, is based on
the work of Levine in \cite{Levine} and \cite{Levine2},
which has been a common wellspring for such results.
This method relies on identifying time invariant spaces,
then proving that solutions in this invariant space must
blowup based on consistent negative concavity and slope
of a functional based on this solution. As mentioned, we
employ this method to classes of solutions of the
Klein-Gordon equation, though in more generality
than \cite{Yangxu} and \cite{Wang} by considering all
metrics that have a certain restriction on the expanding
factor. The fact that there must be restrictions on the
expanding factor should satisfy intuition that aggressively
expanding spacetimes would create problems with local
existence for fluids and would smooth out initial data
too quickly to allow for blowup. Furthermore, this method
requires some assumptions on the initial data relative
to the initial energy which is necessary for the concavity
argument and can be broken down into cases
(listed in Table 1 below). This table is adapted
from Yang \& Xu in \cite{Yangxu} to illustrate our
results as well as the remaining unsolved problems
in the context of cosmological spacetimes.
Note that our assumption (\ref{preconditions}) follows directly from Cases II \& III.\\
%{\textcolor{red}{It seems that the Yang \& Xu method solves two cases as the conditions on $||u_0||^2$ don't play a role in the argument, it is not necessary. This was one of the key points that allowed them to expand Wang's result as Wang only  solved one case by explicitly making the assumption on $||u_0||^2$ and using it in the argument. Yang \& Xu stated on their table that they solved two cases without ever making any assumption on $||u_0||^2$ in the paper...to my knowledge, so I think that we similarly have solved two cases.}}
\\ \\
Table 1:
Obtained results and unsolved problems.\\

\begin{tabular}{|c|c|c|c|c|}
\hline \multirow[b]{2}{*}{ Case I } & \multicolumn{4}{|c|}{ Initial data leading to high energy blowup of problem (\ref{KGF}) } \\
\hline & $I\left(u_0\right)<0$ & $\left\|u_0\right\|^2 \geq \frac{3(\epsilon+2)}{\tilde{m}^2\epsilon} E(t_0)$ & $\frac{3(\epsilon+2)}{\tilde{m}^2\epsilon} E(t_0)>\left(u_0, u_1) \right. >0$ & Still open  \\
\hline Case II & $I\left(u_0\right)<0$ & $\left\|u_0\right\|^2 \geq \frac{3(\epsilon+2)}{\tilde{m}^2\epsilon} E(t_0)$ & $\left(u_0, u_1\right) \geq \frac{3(\epsilon+2)}{\tilde{m}^2\epsilon} E(t_0) >0 $ & Solved in this paper\\
\hline Case III & $I\left(u_0\right)<0$ & $\frac{3(\epsilon+2)}{\tilde{m}^2\epsilon} E(t_0)>\left\|u_0\right\|^2$ & $\left(u_0, u_1\right) \geq \frac{3(\epsilon+2)}{\tilde{m}^2\epsilon} E(t_0) >0$ & Solved in this paper \\
\hline Case IV & $I\left(u_0\right)<0$ & $\frac{3(\epsilon+2)}{\tilde{m}^2\epsilon} E(t_0)>\left\|u_0\right\|^2$ & $\frac{3(\epsilon+2)}{\tilde{m}^2\epsilon} E(t_0)>\left(u_0, u_1\right.) >0$ & Still open  \\
\hline
\end{tabular}\\

As our work follows directly from the method in \cite{Yangxu}, we do not have any specific conditions on the $||u_0||^2$ term. Their process removed specific consideration of this term that had been present in \cite{Wang}, and so covered the Case II and Case III with this value having no specific relationship to the bounding constant. So while a lower bound on this quantity was required for \cite{Wang}, our method has no such requirement and therefore covers the two cases.
\section{Setup and main theorem}

Let $n\in\mathbb{N}$ and $(\mathcal{M},{\mathbf{g}})$
be a Friedmann-Lemaître-Robertson-Walker (FLRW) spacetime
with topology $\mathbb{R}^+\times\mathbb{R}^n$ and metric
\begin{equation}
    {\mathbf{g}} := -dt^2 + a(t)^2\big(dx_1^2+dx_2^2+...+dx_n^2\big) .
    \label{mFLRW}
\end{equation}
Let $m\neq 0$ and $\Box_g$ be the covariant wave operator
associated to $\mathbf{g}$ defined by 
\begin{equation}
\square_{{\mathbf{g}}}u := \frac{1}{\sqrt{|{\mathbf{g}}|}}\partial_{\alpha}\left({\mathbf{g}}^{\alpha\beta} \sqrt{|{\mathbf{g}}|} \partial_{\beta}u\right ) ,
\end{equation}
where $|{\mathbf{g}}|=-\det({\mathbf{g}}_{\alpha\beta})$, ${\mathbf{g}}^{\alpha\beta}$ are the components of the inverse of ${\mathbf{g}}_{\alpha\beta}$, and
greek indices run from $0$ to $n$. For the FLRW metric~\eqref{mFLRW},we have
\begin{equation}
    \square_{{\mathbf{g}}}u=-u_{tt}-
    n\frac{{\dot a}}{a} u_t+ \frac{1}{a^2}\Delta u , 
\end{equation}
with $\Delta$ denoting the
Laplacian operator on the $n$-dimensional flat metric 
\[\Delta :=\sum_{i=1}^{n}\frac{\partial^2}{\partial x^2_i}.\]
Let $0<t_0<T$, $t\in [t_0,T)$,
and $x\in\mathbb{R}^n$. In this work we study solutions
to the Cauchy problem for the
nonlinear Klein-Gordon equation
    \begin{equation} \label{KGF}
    	\begin{cases}
	&- \square_{{\mathbf{g}}}u + m^2u=f(u) \\
        &u(t_0,x) = u_0(x),
       \  u_t(t_0,x)=u_1(x).
	\end{cases}
	\end{equation}
where $u_0\in H_0^1(\mathbb{R}^n)$ and $u_1\in L^2(\mathbb{R}^n)$. This
Cauchy problem describes a local self-interaction for a massive scalar field.
In quantum field theory the matter fields are described by a
function $u:[t_0,T)\times \mathbb{R}^n \rightarrow \mathbb{R}$ that must satisfy equations of motion. In
the case of a massive scalar field, the equation of motion is the semilinear
Klein-Gordon equation in \eqref{KGF}.\\

We now specify the assumptions in our work.

\subsection{Admissible nonlinearities}

For the function $f:\mathbb{R}\rightarrow \mathbb{R}$,
we will assume
that there exists $ \epsilon > 0$
such that
%and such that $\forall u\in\mathbb{R}^n$,
\begin{equation} \label{fepsilon}
    f(s)s \geq (2+\epsilon)F(s) \qquad\mbox{ and }
    \qquad F(s):=\int_0^sf(\xi) \ d\xi.
\end{equation}
%With our equation (\ref{KGF}) associated to the FLRW metric
%${\mathbf{g}}$, of course defined as:
%

Let us consider an example of a type of nonlinearity that satisfies 
condition (\ref{fepsilon}): Take $f(u)=u^p$ with $p>1$; it follows that
\begin{equation*}
     s^{p} s \geq (2+\epsilon)\frac{s^{p+1}}{p+1}
     \qquad \Rightarrow \qquad p+1 \geq (2+\epsilon) .
\end{equation*}
Hence power nonlinearities with $p\geq 1+\epsilon$
will satisfy Eq.(\ref{fepsilon}). Another example of
an admissible nonlinearity is the case of a {\it{focusing}}
power nonlinearity $f(u)=|u|^{p-1}u$. The
{\it{defocusing}} power nonlinearity $f(u)=-|u|^{p-1}u$
is a {\it{non-example}} of an admissible nonlinearity.

Next we describe assumptions on the admissible nonlinearities that are
needed to ensure local wellposedness for problem \eqref{KGF}.\\

\noindent {\bf{Definition:}} A function $h:=h(x,s)$,
$h:\mathbb{R}^n\times \mathbb{R}\rightarrow\mathbb{R}$
is said to be {\it{Lipshitz continuous in $s$ with exponent $\alpha$}},
if there exist $\alpha\geq 0$ and $C>0$ such that
\[|h(x,s)-h(x,v)|\leq C|s-v|(|s|^{\alpha}+|v|^{\alpha})
\qquad {\text{for all }}s,v\in\mathbb{R}, x\in\mathbb{R}^n . \]    

In our work we will assume that $f(s)$
satisfies the following two assumptions:
\begin{enumerate}[1)]
    \item  $f(0)=0$,
    \item $f$ is Lipshitz continuous with exponent $0\leq \alpha\leq \frac{2}{n-2}$.
\end{enumerate}

The restriction on the exponent $\alpha$
is needed for local wellposedness for problem
\eqref{KGF} to hold. We discuss this in more detail in Section 3.

\subsection{Assumptions on the expanding factor $a(t)$}

We consider cosmologies undergoing an accelerated expansion
in the direction of positive time $t$. This expansion corresponds to 
\begin{equation}\label{accel_exp}
 a(t)>0 \qquad {\text{and}}\qquad
 \dot a(t):=\frac{da}{dt}>0 \ ,
\end{equation}
for all $t>0$. Let $a_0:=a(t_0)$, $\dot{a}_0:=\dot{a}(t_0)$; we assume
that there exists a $t_0>0$ such that
\begin{equation}\label{initial_time}
	\frac{\dot{a}(t)}{a(t)}\leq
	\frac{\dot{a}_0}{a_0}\leq \frac{\epsilon}{6n} \ ,
\end{equation}
holds for all $t\geq t_0$. This assumption is essential in order
to treat the effect of the accelerated expansion perturbatively.

Next we codify an assumption
on the mass $m$ and the expanding factor
$a(t)$ that is needed for local wellposedness for
problem \eqref{KGF} to hold.
Define the {\it{curved mass $M(t)$}} (see \cite{Gal1}) to be
\begin{equation}\label{eff_mass}
    M^{2}(t):=m^2+\Big(\frac{n}{2}-\frac{n^2}{4}\Big)
    \Big(\frac{\dot{a}(t)}{a(t)}\Big)^2-\frac{n}{2}
    \frac{\ddot{a}(t)}{a(t)}.
\end{equation}
We make the following assumption on $M(t)$:
there
exists a $c_0>0$ such that
\begin{equation}\label{curved_mass_cond}
  M(t)>c_0>0, \qquad \dot{M}(t)\leq 0, \qquad
{\text{for all }}t\in[t_0,\infty)  .
\end{equation}
This assumption will be
satisfied by the FLRW and de Sitter metrics
under suitable conditions on the parameters.
See Section 2.4
for more details.

\subsection{Norms, energy, and modified Nehari
functional.}

We denote the $L^2$ norm and inner product by
\begin{align*}
    &\|u\|:=\|u\|_{L^{2}(\mathbb{R}^n)}
    :=\Big(\int_{\mathbb{R}^{n}}u^{2}
    d x\Big)^{1/2} \ , \qquad 
    (u,v):=\int_{\mathbb{R}^n}(uv)
    \ d x \ .
\end{align*}
We define the functional spaces
\begin{align*}
	H^{1}&:=H^{1}(\mathbb{R}^n):=\Big\{u\in L^2 \ \big| \  \|u\|_{H^1}
	=\|(1-\Delta)^{1/2}u\|<+\infty \Big\},\\
	H_0^{1}&:=H_0^{1}(\mathbb{R}^n):=\Big\{u\in H^1 \ \big| \ {\text{ supp$(u)$ is
	compact in $\mathbb{R}^n$}} \Big\}.
\end{align*}
The energy for this problem is defined to be
\begin{equation} \label{energydef}
    E(t) := \frac{1}{2}\left(\| u_t\|^2 +m^2\| u \|^2
    + \frac{1}{a^2}\|\nabla u \|^2 -
    2\int_{\mathbb{R}^n}^{} F(u) \  d x \right) ,
\end{equation}
where $\nabla u$ denotes the spatial gradient of $u$. We also define the {\it{modified Nehari functional}}
\begin{equation}
   I(u) := \int_{\mathbb{R}^n}\left(m^2u^2 + 
   \frac{1}{a^2}(\nabla u)^2 + n\frac{\Dot{a}}{a}\Tilde{m}^2uu_t 
   - uf(u)\right)d x.
\end{equation}
Note that because of $u_t$, the modified Nehari functional $I(u)$
is defined on functions that depend on time. Since $I(u)$
is a function of time $t$, we will use the notation $I(u(t))$ when we wish to emphasize
this fact. Now we define the {\it{unstable set for the Nehari functional}}
\begin{equation} \label{unstable}
    \mathcal{B} := \{u \in C([t_0,T);H_0^1(\mathbb{R}^n))\cap C^1([t_0,T);L^2(\mathbb{R}^n)) \ |\ I(u)<0\} . 
\end{equation}

The functional space in this definition
comes directly from the regularity
given by local wellposedness
theorem in section 3. The key here is that our assumptions
guarantee that the sign condition
$I(u(t))<0$ will be preserved by the flow. This fact is proved in section 6.1.

In the following, we will work as if $u$
is smooth and supported away from spatial infinity.
This assumption can be removed by a standard approximation argument.

\subsection{Statement of main theorem}
We are now ready to state our main result.

\begin{theorem}\label{main_result}
Suppose that there exists $ \epsilon > 0$ such that Eq. \eqref{fepsilon}
is satisfied. Assume that $f(0)=0$ and that $f(s)$ is Lipshitz continuous
with exponent $0\leq \alpha\leq \frac{2}{n-2}$. Suppose
that the expanding factor $a(t)$ satisfies
Eq. \eqref{accel_exp} and that the initial time $t_0$ satisfies
Eq. \eqref{initial_time}. Assume that $m>0$, and let $\tilde{m}=min\{1,m\}$. Further, assume that there is a positive number $c_0$
such that the curved mass $M(t)$ satisfies \eqref{curved_mass_cond}. Let $u_0(x)\in H_0^1$ 
and $u_1(x)\in L^2$. Suppose that $u_0,u_1$ are such that $I(u(t_0))<0$ and that
\begin{equation} \label{preconditions}
    (u_0,u_1)\geq\frac{2(\epsilon + 2)}{\tilde{m}^2\big(\epsilon-2n\big(\frac{\Dot{a}_0}{a_0}\big)\big)}
    E(t_0) = \frac{3(\epsilon+2)}{\tilde{m}^2\epsilon}E(t_0)>0 \ ,
\end{equation}
holds for the initial data. Then the corresponding local solution to (\ref{KGF}) blows 
up in finite time. Furthermore, the maximal time 
of existence $T_{max}$ has the explicit
upper bound
\[T_{max}<t_0+\frac{B(t_0)}{\epsilon B'(t_0)}
 \ , \]
 with
 \[B(t):=\|mu\|^2 \ , \qquad
 B'(t):=2(m^2u,u_t) \ .\]
\end{theorem}

One of the advantages of this result is the
simple, explicit form for the upper bound $T_{max}$;
from
this form one can readily infer how
each parameter affects the time of existence. For instance,
larger $\epsilon$ results in shorter
time of existence. Similarly, larger
inner product
$B'(t_0):=2(m^2u(t_0),u_t(t_0))$ also results in shorter
time of existence. From our main theorem
we can also read off how some
commonly referenced cosmological
backgrounds affect $T_{max}$: For example, for FLRW
with power expansion $a(t)=t^k$, we have
    \[\frac{\dot{a}(t)}{a(t)}=\frac{k}{t}\]
and therefore $t_0=6nk/\epsilon$.
Patching this together with the local
existence result means that as $k$
increases, we have a larger and larger
region of safety. So in this
case, the faster the expansion, the longer
it takes for the nonlinear effects to kick
in. By contrast, in the cases of
Minkowski spacetime $a(t)=1$ or de Sitter $a(t)=e^{Ht}$,
the choice of initial time $t_0>0$ has no effect
on the upper bound $T_{max}$.

As mentioned in the introduction, this result is complementary
to the result in \cite{Gal1}. In that work, the authors
prove small data global existence for 
a class of semilinear Klein-Gordon equations propagating
in a family of expanding cosmological
backgrounds which include de Sitter. The main point of departure
from that work is the fact that here we are considering {\bf{large}} initial data
in $H_0^1\times L^2$. This comes from the fact that our initial condition
\eqref{preconditions} excludes the possibility of small initial data in
these functional spaces. In particular, our initial conditions
imply that the norm of initial data cannot be too small
due to the existence of a positive lower bound for the $H^1$ norm
at the initial time $t_0$. For the interested reader, we
prove this fact in an Appendix at the end of this manuscript.

\subsection{Applicable cosmology examples}
Two of the most commonly referenced cosmological
backgrounds have expanding factors that allow the
use of our method. For certain
combinations of parameters, FLRW space with
expanding factor $a(t)= t^k$, $k>0$ works as
long as our initial time is spaced 
far enough forward from the singularity at $t=0$.
De Sitter space works as long as the Hubble constant $H> 0$
is suitably low.
We present the relevant computations here.\\

\subsubsection{De Sitter} For the de Sitter metric, we have $a(t)= e^{Ht}$, with $H>0$. In this case we get
\begin{equation*}
    \frac{\dot{a}(t)}{a(t)} = H.
\end{equation*}
Thus we require that $H<\frac{\epsilon}{6n}$ for our method to be applicable. Next we compute
the curved mass
\begin{align*}
    &M^{2}(t)=m^2+\Big(\frac{n}{2}-\frac{n^2}{4}\Big)
    H^2-\frac{n}{2}H^2=\Big(m-\frac{n}{2}H\Big)
    \Big(m+\frac{n}{2}H\Big),\\
    &\dot{M}(t)=0,
\end{align*}
so that condition \eqref{curved_mass_cond} is satisfied if
\[m-\frac{n}{2}H>0.\]
In particular, our method
is applicable
if $H$ is sufficiently small and $m$ is sufficiently large.

\subsubsection{FLRW}
For the Friedmann–Lemaître-Robertson–Walker metric with power 
expansion we have $a(t)= t^k$, 
$k>0$. Thus
\begin{equation*}
    \frac{\dot{a}(t)}{a(t)} = 
    \frac{k}{t}.
\end{equation*}
So to satisfy condition \eqref{initial_time}
we choose
\begin{equation}
    t_0=\frac{6kn}{\epsilon}.
\end{equation}
Next we compute the curved mass
\begin{align*}
    &M^{2}(t)=m^2+\Big(\frac{n}{2}-\frac{n^2}{4}\Big)
    \Big(\frac{k}{t}\Big)^2-\frac{n}{2}\frac{k(k-1)}{t^2}=m^2+\frac{nk}{2t^2}\Big(
    1-\frac{nk}{2}\Big)
\end{align*}
from which we can see that 
condition \eqref{curved_mass_cond} 
is satisfied if $0<k\leq \frac{2}
{n}$, for example.

\section{Local Wellposedness}

In this section we make use of the local wellposedness result in \cite[Section 1]{Gal1}. We state
the theorem, hypotheses, and (readjusted) notation for the reader's convenience here.

\begin{theorem}[Galstian and Yagdjian, 2014]
Suppose that the expanding factor $a(t)$ satisfies
assumptions \eqref{accel_exp} and \eqref{initial_time}. Suppose that $m>0$ and that there is a positive number $c_0$
such that the curved mass $M(t)$ satisfies \eqref{curved_mass_cond}. Consider the Cauchy problem for the equation
    \begin{equation} \label{KGF_Gal}
    	\begin{cases}
	&u_{tt} + n\frac{\dot a}{a}u_t-a^{-2}\Delta u + m^2u-V_{u}^{'}(x,t,u)=0 \\
        &u(t_0,x) = u_0(x),
       \  u_t(t_0,x)=u_1(x).
	\end{cases}
	\end{equation}
 where $V_{u}^{'}(x,t,u)=-\Gamma(t)f(u)$
 with $\Gamma\in C([t_0,\infty))$, $f$ Lipshitz continuous with exponent
 $0\leq \alpha\leq \frac{2}{n-2}$ and $f(0)=0$. Then for every
 $u_0(x)\in H^1(\mathbb{R}^n)$ and $u_1(x)\in L^2(\mathbb{R}^n)$ there exists $T_1>t_0$ such that the problem
 \eqref{KGF_Gal} has a unique solution $u\in C([t_0,T_1);H_0^1(\mathbb{R}^n))\cap C^1([t_0,T_1);L^2(\mathbb{R}^n))$
\end{theorem}

The statement of this theorem, combined with our assumptions in Section 2 and with the fact that
$H_{0}^1(\mathbb{R}^n)\subset H^1(\mathbb{R}^n)$ immediately implies that
the Cauchy problem \eqref{KGF} satisfies the conditions of this theorem with $\Gamma(t)=1$.

\section{Energy Formalism}
In this section we develop
the framework needed to prove
our energy estimate. We largely
follow the setup of \cite{Oliver}
and \cite{Oliver2}. Define the
{\it{energy-momentum tensor}} to be
\begin{equation}
    T_{\alpha\beta} := \partial_\alpha u \partial_\beta u
    - \frac{1}{2}g_{\alpha\beta}
    \Big(\partial^\mu u\partial_\mu u + m^2 u^2 - 2 F(u)  \Big)\;.
\end{equation}
Let $D$ be the covariant derivative corresponding to the metric $g$. The energy momentum tensor is divergence free
\[D^\alpha T_{\alpha\beta}= 0 .\] Given a (smooth) 
%Greg will check the divergence identity with the definition of the PDE.
vector field $X$ we define the 1-form
\[{}^{(X)}P_\alpha := T_{\alpha\beta}X^\beta .\]
Taking the divergence yields
\begin{equation} \label{divergence}
    D^\alpha \big({}^{(X)}P_\alpha \big) = \frac{1}{2}\;^{(X)}\pi^{\alpha\beta}T_{\alpha\beta}
\end{equation}
where the symmetric 2-tensor $^{(X)}\pi$ is
the \textit{deformation tensor} of
$g$ with respect to $X$. Integrating this divergence
over the spacetime region
\[\{(t,x) \;|\;t_0\leq t\leq t_1\}\] and using Stokes'
theorem we get the following {\it{multiplier identity}}
\begin{equation} \label{multident}
\begin{split}
    \int_{t=t_0}\; ^{(X)}P_\alpha N^\alpha
    |{\mathbf{g}}|^{\frac{1}{2}} \ dx -
    \int_{t=t_1}\; ^{(X)}P_\alpha N^\alpha
    |{\mathbf{g}}|^{\frac{1}{2}} \ dx 
    = \int_{t_0}^{t_1}\int_{\mathbb{R}^n}(\frac{1}{2} \;^{(X)}\pi^{\alpha\beta}T_{\alpha\beta}) \ dV_g
\end{split}
\end{equation}
where
\[dV_g := |{\mathbf{g}}|^{\frac{1}{2}}dtd x\]
and $N$ is the vector field defined below:
\begin{equation}
    N^\alpha := \frac{-g^{\alpha\beta}\partial_\beta t}
    {(-g^{\alpha\beta}\partial_\alpha t \partial_\beta t)^{\frac{1}{2}}}
    = -g^{\alpha 0} .
\end{equation}
The integrand on the left hand side
of Eq. (\ref{multident}) is the {\it{energy density}}
associated to $X$ through the foliation by the spacelike hypersurfaces $t=const$.
\newline\newline
Given a vector field $X$ we define the
\textit{normalized deformation tensor} of $X$ to be
\begin{equation}
    ^{(X)}\hat{\pi} = \;^{(X)}\pi -
    \frac{1}{2}|{\mathbf{g}}|\cdot trace(^{(X)}\pi)
\end{equation}
which can be computed using the following identity (see \cite{Oliver}):
\begin{equation} \label{NormDTens}
    ^{(X)}\hat{\pi}^{\alpha\beta} = -|{\mathbf{g}}|^{-\frac{1}{2}}
    X(|{\mathbf{g}}|^\frac{1}{2} g^{\alpha\beta}) - g^{\alpha\beta}\partial_\gamma X^\gamma + g^{\alpha\gamma}\partial_\gamma X^\beta + g^{\beta\gamma}\partial_\gamma X^\alpha .
\end{equation}
In combination with Eq. \eqref{divergence}, we get
\begin{equation} \label{DPXhat}
    D^\alpha \big({}^{(X)} P_\alpha\big)
    = \frac{1}{2} \;^{(X)}\hat{\pi}^{\alpha\beta}\partial_\alpha u \partial_\beta u .
\end{equation}
Combining (\ref{DPXhat}) and (\ref{multident}) yields
\begin{multline} \label{multident2}
    \int_{t=t_0}\; ^{(X)}P_\alpha N^\alpha
    |{\mathbf{g}}|^{\frac{1}{2}}d x -
    \int_{t=t_1}\; ^{(X)}P_\alpha N^\alpha
    |{\mathbf{g}}|^{\frac{1}{2}}d x 
    = \int_{t_0}^{t_1}\int_{\mathbb{R}^n}
   \frac{1}{2} \Big( \;^{(X)}\hat{\pi}^{\alpha\beta}\partial_\alpha u 
    \partial_\beta u\Big)  |{\mathbf{g}}|^{\frac{1}{2}}dtd x \\
\end{multline}

\section{Proof of Energy Inequality}
In this section we will prove the energy nonincreasing
property as stated below:
\begin{lemma}
Let $u(t,x)$ be a solution to problem \eqref{KGF}. Then
\[ E(t_0) \geq E(t) \qquad  \;\forall t \geq t_0\]
\end{lemma}

\noindent \textbf{Proof:} Let $X = a^{-n}\partial_t$. We can compute the 
normalized deformation tensor of $X$ using Eq.
(\ref{NormDTens}). This yields
%\textbf{Case I:}
\begin{align*}
%   \begin{split}
    ^{(X)}{\hat{\pi}^{0 0}} &= 2 n a^{-n-1}\Dot{a}  \geq 0, \hspace{1.1in}
%   \end{split}
%\end{equation}
%\textbf{Case II:}
%\begin{equation} \label{NDTens}
     ^{(X)}{\hat{\pi}^{0 j}} =  0, \\
% \end{equation}
%\textbf{Case III:}
%\begin{equation} \label{NDTens}
%   \begin{split}
     ^{(X)}{\hat{\pi}^{i j}} &=   (2-n) (a^{-n-3}\Dot{a})\delta^{i j} + n a^{-n-3} \Dot{a} \;\delta^{i j}= 2 a^{-n-3}\Dot{a} \;\delta^{i j}  \geq 0. 
%   \end{split}
 \end{align*}

 Hence, 
 \begin{equation}
     \frac{1}{2} \big(\;^{(X)}\hat{\pi}^{\alpha\beta}\partial_\alpha u \partial_\beta u
     \big)\ |{\mathbf{g}}|^{\frac{1}{2}}
     =  n\frac{\Dot{a}}{a}|u_{t}|^2+n\frac{\Dot{a}}{a^3} |\nabla u|^2 \geq 0 . 
 \end{equation}
From Eq. (\ref{multident2}) it follows that:
\begin{equation}\label{final_energy}
    \begin{split}
       \int_{t=t_0}\; ^{(X)}P_\alpha N^\alpha |
       {\mathbf{g}}|^{\frac{1}{2}}dx &- \int_{t=t_1}\; ^{(X)}P_\alpha N^\alpha |{\mathbf{g}}|^{\frac{1}{2}}dx 
       \geq 0 .  %\\ 
%     \implies 
%     \int_{t=t_0}\; ^{(X)}P_\alpha N^\alpha d^{\frac{1}{2}}dx &\geq \int_{t=t_1}\; ^{(X)}P_\alpha N^\alpha d^{\frac{1}{2}}dx \\
    \end{split}
\end{equation}

Expanding the integrand yields
\[ ^{(X)}P_\alpha N^\alpha |{\mathbf{g}}|^{\frac{1}{2}}
=-g^{\alpha 0}T_{\alpha0}a^{-n} a^{n}=T_{00}=
\frac{1}{2}\Big(|u_t|^2+a^{-2}|\nabla u|^2
+m^2 u^2-2F(u)\Big)\]

Combining this with Eq. \eqref{final_energy} gives us
%\[   N^\alpha = \frac{-g^{\alpha 0}}
%{(-g^{0 0} )^{\frac{1}{2}}}\]
%Since $t_1$ is arbitrary this inequality holds $\forall t \geq t_0$
\begin{equation} \label{enineq}
     E(t_0) \geq E(t) \;\;\;\; \forall t \geq t_0. \;\;\;\;\;\;\;\square
\end{equation}

\section{Proof of Blow up in Finite Time}
%\textit{Note:} In keeping with the method followed in \cite{Yangxu},the blowup proof will proceed with the assumption that $m=1$. One can complete the proof without this change of units, keeping in mind a factor of $\frac{1}{m^2}$ will show up in the lower bound of $(u_0,u_1)$. This means we can keep the bound as is for $m\geq 1$, and the bound will be pushed upward for all $m<1$. This factor will also cause the method to break down for $m=0$, so this method will not be applicable to prove a blowup result for the Wave equation.\newline
In this section we will prove that
the set $\mathcal{B}$ defined in Eq. (\ref{unstable})
is invariant for the Cauchy problem when the assumptions in the
main theorem are satisfied.
\subsection{Proof that $\mathcal{B}$ is an invariant set}

\begin{lemma}\label{invariant}
Let $u_0(x)\in H_0^1$ and $u_1(x)\in L^2$ and $a_0= a(t_0)$ and $\tilde{m}=min\{1,m\}$. Then
$u,u_t$ of the corresponding local solution to \eqref{KGF} satisfy $I(u(t))<0$
for all $t \leq T_{max}$ provided that $I(u(t_0))<0$ initially and
\begin{equation} 
    (u_0,u_1)\geq\frac{2(\epsilon + 2)}{\tilde{m}^2\big(\epsilon-2n\frac{\Dot{a}_0}{a_0}\big)}E(t_0) = \frac{2(\epsilon + 2)}{\frac{2}{3}\tilde{m}^2\epsilon}E(t_0) = \frac{3(\epsilon+2)}{\tilde{m}^2\epsilon}E(t_0)>0 . 
\end{equation}
\end{lemma}
\textbf{Proof:} For the purpose of a contradiction, we assume $\exists t^*\in (t_0,T_{max})$ where $t^*$ is the least positive value such that:
\begin{equation}
    I(u(t^*)) = 0
\end{equation}
And we therefore have:
\begin{equation} \label{nehnegreg}
    I(u(t)) < 0 \;\; \forall t \in [t_0,t^*)
\end{equation}
We define the auxiliary function:
\begin{equation}
    B(t) := \|mu\|^2
\end{equation}
and therefore
\begin{equation}
    B'(t)=2(m^2u,u_t)
\end{equation}
and we therefore arrive at the fact:
\begin{equation} \label{firstderpos}
    B'(t_0) = 2m^{2}(u_0,u_1) \geq \frac{6(\epsilon+2)}{\epsilon}E(t_0)>0
\end{equation}
and from (\ref{KGF})
\begin{equation} \label{secderdef}
    \begin{split}
        B''(t) &= 2m^2(u,u_{tt})+2m^2\|u_t\|^2 \\
        &= 2m^2\|u_t\|^2 + 2m^2\int_{\mathbb{R}^n}u(a^{-2}\Delta u - n\frac{\Dot{a}}{a}u_t+f(u)-m^2u)d x \\
        &= 2m^2\|u_t\|^2 + 2m^2\int_{\mathbb{R}^n}(-a^{-2}(\nabla u)^2 - n\frac{\Dot{a}}{a}uu_t+uf(u)-m^2u^2)d x \\
        &\geq 2m^2\|u_t\|^2 - 2m^2I(u)
    \end{split}
\end{equation}
So from (\ref{nehnegreg}) and (\ref{secderdef}) we can establish:
\begin{equation} \label{secderpos}
    B''(t)>0 \;\; \forall t \in [t_0,t^*)
\end{equation}
So from (\ref{firstderpos}) and (\ref{secderpos}) we can conclude:
\begin{equation}
    B'(t) > B'(t_0) = 2m^2(u_0,u_1)\geq \frac{6(\epsilon+2)}{\epsilon}E(t_0)>0
\end{equation}
With this we have established that $B(t)$ and $(u,u_t)$ are positive and monotonic increasing on $[t_0,t^*)$, i.e.
\begin{equation}\label{innereq}
    (u(t),u_t(t))\geq (u_0,u_1)\geq\frac{3(\epsilon +2)}{\tilde{m}^2\epsilon}E(t_0)>0, \;\;  \forall t \in (t_0,t^*)
\end{equation}
From this result we have $(u(t^*),u_t(t^*))\geq\frac{3(\epsilon+2)}{{\tilde{m}^2}\epsilon}E(t^*)$, but now we proceed using (\ref{energydef})  along the same lines as \cite{Yangxu}, but with our modified energy inequality (\ref{enineq}):

\begin{equation} \label{eq1}
\begin{split}
E(t_0)&\geq E(t)\\
&=\frac{1}{2}\| u_t\|^2 + \frac{1}{2}m^2\| u \|^2 + \frac{1}{2}a^{-2}\|\nabla u \|^2 - \int_{\mathbb{R}^n}^{} F(u) d x  \\
&\geq \frac{1}{2}\| u_t\|^2 + \frac{1}{2}m^2\| u \|^2 + \frac{1}{2}a^{-2}\|\nabla u \|^2 - \frac{1}{\epsilon + 2} \int_{\mathbb{R}^n} u f(u) d x\\
&=\frac{1}{2}\| u_t\|^2 + (\frac{1}{2} - \frac{1}{\epsilon+2})(m^2\|u\|^2 + a^{-2}\|\nabla u\|^2) + \frac{1}{\epsilon + 2}I(u)\\
&\;\;\;\;-\frac{n\Dot{a}}{a(\epsilon+2)}(\tilde{m}^2u_0,u_1)
\end{split}
\end{equation}
\newline
Let $a_{*} = a(t^{*})$. Therefore, as we know $I(u(t^*))=0$ from our assumption, we arrive at:
\begin{equation} \label{eq2}
\begin{split}
E(t_0) &\geq \frac{1}{2}\|u_t(t^*)\|^2+(\frac{1}{2} - \frac{1}{\epsilon+2})(m^2\|u(t^*)\|^2 + a^{-2}\|\nabla u(t*)\|^2)\\
&\;\;\;\;-\frac{n\Dot{a_*}}{a_*(\epsilon + 2)}(\tilde{m}^2u(t^*),u_t(t^*))\\
&= \frac{1}{2}\|u_t(t^*)\|^2+( \frac{\epsilon}{2(\epsilon+2)})(m^2\|u(t^*)\|^2 + a^{-2}\|\nabla u(t^*)\|^2)\\
&\;\;\;\;-\frac{n\Dot{a_*}}{a_*(\epsilon + 2)}(\tilde{m}^2u(t^*),u_t(t^*))\\
&\geq \frac{\epsilon}{2(\epsilon+2)}(\|u_t(t^*)\|^2+\|mu(t^*)\|^2 + a^{-2}\|\nabla u(t^*)\|^2)\\
&\;\;\;\;-\frac{n\Dot{a_*}}{a_*(\epsilon + 2)}(\tilde{m}^2u(t^*),u_t(t^*))\\
&\geq \frac{\epsilon}{2(\epsilon+2)}(2(m^2u(t^*),u_t(t^*)) + a^{-2}\|\nabla u(t^*)\|^2)\\
&\;\;\;\;-\frac{n\Dot{a_*}}{a_*(\epsilon + 2)}(\tilde{m}^2u(t^*),u_t(t^*))\\
&\geq \frac{\epsilon}{2(\epsilon+2)}(2(\tilde{m}^2u(t^*),u_t(t^*)) -\frac{n\Dot{a_*}}{a_*(\epsilon + 2)}(\tilde{m}^2u(t^*),u_t(t^*))\\
&= \frac{a_*\epsilon-n\Dot{a_*}}{a_*(\epsilon+2)}(\tilde{m}^2u(t^*),u_t(t^*))\\
&= \frac{\epsilon-n\frac{\Dot{a_*}}{a_*}}{\epsilon+2}(\tilde{m}^2u(t^*),u_t(t^*))\\
%&\geq  \frac{\epsilon-n\frac{\Dot{a_*}}{a_*}}{\epsilon+2}(\tilde{m}^2u(t^*),u_t(t^*))
\end{split}
\end{equation}
Therefore we have: \\
\begin{equation} \label{ineqend}
    \begin{split}
        E(t_0) &> \frac{\epsilon-n\frac{\Dot{a_*}}{a_*}}{\epsilon+2}(\tilde{m}^2 u(t^*),u_t(t^*))\\
        \therefore\;\;\; (\tilde{m}^2u(t^*),u_t(t^*)) &< \frac{\epsilon +2}{\epsilon - n\frac{\Dot{a_*}}{a_*}}E(t_0)\\
        \therefore\;\;\; (\tilde{m}^2u(t^*),u_t(t^*)) &< \frac{\epsilon +2}{\epsilon - n\frac{\Dot{a_0}}{a_0}}E(t_0)\\
         \therefore\;\;\; (\tilde{m}^2u(t^*),u_t(t^*)) &< \frac{\epsilon +2}{\epsilon - \frac{1}{6} \epsilon}E(t_0)\\ 
          \therefore\;\;\; (u(t^*),u_t(t^*)) &< \frac{6(\epsilon +2)}{5\tilde{m}^2\epsilon}E(t_0)
        \end{split}
\end{equation}
\\
Which gives us that $(u(t^*),u_t(t^*))<\frac{6(\epsilon + 2)}{5\tilde{m}^2\epsilon}E(t_0)$,which obviously contradicts \ref{innereq}.
\newline
\newline
Therefore it must follow that no such $t^*$ exists, and therefore $I(u(t))<0$,   $\;\forall t>t_0$ and therefore the set $\mathcal{B}$ is invariant. $\square$
\subsection{Proof of the Main Theorem}
In this section we prove Theorem \ref{main_result}.\\

\textbf{Proof:} From Lemma \ref{invariant} we know that $u(x,t)\in\mathcal{B}$ implies
\begin{equation}
    B(t)>0 \;\; \forall t\in [t_0,T_{max})
\end{equation}
And from (\ref{firstderpos}) and the Cauchy-Schwarz Inequality we arrive at
\begin{equation}
    B'(t)^2=4m^4(u,u_t)^2 \leq 4m^4\|u\|^2\|u_t\|^2 = 4m^2B(t)\|u_t\|^2
\end{equation}
Combining with (\ref{secderdef}) we get the inequality
\begin{equation} \label{specialineq1}
    \begin{split}
        &B''(t)B(t)-\frac{\lambda +3}{4}B'(t)^2\\
        \geq &B(t)(B''(t)-(\lambda+3)m^2\|u_t\|^2)\\
        \geq &B(t)(-(\lambda+1)m^2\|u_t\|^2-2m^2I(u))
    \end{split}
\end{equation}
From our energy inequality (\ref{eq1}) we have:
\begin{align*}
E(t_0) \geq \frac{1}{2} \|u_t\|^2 + (\frac{1}{2}-\frac{1}{\epsilon +2})(m^2\|u\|^2+ a^{-2}\ |\nabla u\|^2)+\frac{1}{\epsilon +2}I(u)-\frac{n\Dot{a}}{a(\epsilon+2)}\tilde{m}^2(u,u_t)
\end{align*}
Therefore,
\begin{align*}
E(t_0) - \frac{1}{2} \|u_t\|^2 - \frac{\epsilon}{2(\epsilon +2)}(m^2\|u\|^2+a^{-2}\|\nabla u\|^2) + \frac{n\Dot{a}}{a(\epsilon+2)}\tilde{m}^2(u,u_t) \geq \frac{1}{\epsilon +2}I(u)
\end{align*}
Which gives us:
\begin{equation} \label{eq3}
2I(u) \leq 2(\epsilon +2)E(t_0)-(\epsilon + 2)\|u_t\|^2-\epsilon(m^2\|u\|^2+a^{-2}\|\nabla u\|^2) +2n\tilde{m}^2\frac{\Dot{a}}{a}(u,u_t)
\end{equation}
We now define $\xi:\mathbb{R}_{nonneg} \to \mathbb{R}$
\begin{align*}
\xi(t) := -m^2(\lambda +1)\|u_t\|^2 - 2m^2I(u)
\end{align*}
And combining this with inequality (\ref{eq3}) above we get:
\begin{align*}
\frac{\xi(t)}{m^2} &\geq -(\lambda +1)\|u_t\|^2 - 2(\epsilon +2)E(t_0)+(\epsilon + 2)\|u_t\|^2+\epsilon(m^2\|u\|^2+a^{-2}\|\nabla u\|^2)\\
&\;\;\;\;-2n\tilde{m}^2\frac{\Dot{a}}{a}(u,u_t)\\
&=(\epsilon+1-\lambda)\|u_t\|^2+\epsilon(m^2\|u\|^2+a^{-2}\|\nabla u\|^2)-2(\epsilon +2)E(t_0)\\
&\;\;\;\;-2n\tilde{m}^2\frac{\Dot{a}}{a}(u,u_t)
\end{align*}
We now substitute $\lambda = \frac{\epsilon}{2}+1$ to arrive at:
\begin{equation} \label{eq4}
\begin{split}
\frac{\xi(t)}{m^2} &\geq \frac{\epsilon}{2}\|u_t\|^2+\epsilon(m^2\|u\|^2+a^{-2}\|\nabla u\|^2)-2(\epsilon +2)E(t_0)-2n\tilde{m}^2\frac{\Dot{a}}{a}((u,u_t)\\
&> \frac{\epsilon}{2}(\|u_t\|^2+m^2\|u\|^2)-2(\epsilon+2)E(t_0)-2n\tilde{m}^2\frac{\Dot{a}}{a}(u,u_t) \\
&\geq \tilde{m}^2(\epsilon-2n\frac{\Dot{a}}{a})(u,u_t)-2(\epsilon +2)E(t_0)\\
&\geq\frac{2}{3}\epsilon \tilde{m}^2(u,u_t)-2(\epsilon +2)E(t_0)
 \end{split}
\end{equation}
\newline And since we have 
\begin{equation} (u(t),u_t(t))>(u_0,u_1)\geq\frac{3(\epsilon +2)}{\tilde{m}^2\epsilon}E(t_0)>0, \;\;  \forall t \in (t_0,T_{max})
\end{equation}
%Or
%\begin{equation} (u(t),u_t(t))>(u_0,u_1)>0, \;\;  \forall t \in (t_0,T_{max}]
%\end{equation}
We arrive at the conclusion:
\begin{equation} \label{eq6}
\xi(t) > 0, \;\;  \forall t \in (t_0,T_{max})
\end{equation}
Now, we will, using (\ref{eq6}) to continue our inequality from (\ref{specialineq1}) consider:
\begin{equation}
    \begin{split}
&B''(t)B(t)-\frac{\lambda+3}{4}B'(t)^2 \\
\geq &B(t)(B''(t)-(\lambda+3)m^2\|u_t\|^2) \\
\geq &B(t)(-(\lambda +1)m^2\|u_t\|^2 - 2m^2I(u)) \\
= &B(t)\xi(t) >0\\
\because\; &B(t) >0 \;and \; \xi(t)>0, \;\; \forall t \in (t_0,T_{max}]\\
\therefore \;&B''(t)B(t)-\frac{\lambda+3}{4}B'(t)^2>0, \;\; \forall t \in (t_0,T_{max})
\end{split}
\end{equation}
\\We can now rejoin \cite{Wang} in the concavity argument. As $\frac{4+\epsilon}{4} > 1$, and letting $\alpha = \frac{\epsilon}{4}$, we get:
\begin{equation}
    \begin{split}
        (B^{-\alpha})' &= -\alpha B^{-\alpha - 1}B'(t)<0,\\
        (B^{-\alpha})'' &= -\alpha B^{-\alpha - 2}[B''(t)B(t)-\frac{4+\epsilon}{4}B'(t)^2] < 0\\
        &\forall t \in (t_0,T_{max}).
    \end{split}
\end{equation}
And if we consider the line tangent to $B^{-\alpha}$ at $t = t_0$:
\begin{equation}
    G_{B^{-\alpha}(t_0)}(t):= -\alpha B^{-\alpha-1}(t_0)B'(t_0)(t-t_0) + B^{-\alpha}(t_0)
\end{equation}
And solve for $G_{B^{-\alpha}(t_0)}(t) = 0$, we get:
\begin{equation}
\begin{split}
    0 &= -\alpha B^{-\alpha-1}(t_0)B'(t_0)(t-t_0) + B^{-\alpha}(t_0)\\
    \therefore \;  B^{-\alpha}(t_0) &= \alpha B^{-\alpha-1}(t_0)B'(t_0)(t-t_0)\\
    \therefore \; t &= \frac{B^{-\alpha}(t_0)}{\alpha B^{-\alpha-1}(t_0)B'(t_0)} + t_0 = t_0 +\frac{4B(t_0)}{\epsilon B'(t_0)}
\end{split}
\end{equation}
Therefore it follows that $B^{-\alpha}$ is concave down and decreasing and subsequently bounded above by $G_{B^{-\alpha}(0)}(t)$. And since we know $B(t_0)>0$ it must follow that $\exists T_{max} < \frac{4 B(t_0)}{\epsilon B'(t_0)}$ such that:
\begin{equation}
\begin{split}
    \lim_{t\to T_{max}^-} B^{-\alpha} &= 0\\
    \therefore \;\lim_{t\to T_{max}^-} B(t) &= \lim_{t\to T_{max}^-} \|u(t,.)\|^2 \to \infty \;\;\;\;\;\square.
\end{split}
\end{equation}
\begin{figure}[h!]
\centering
\includegraphics[scale=.4]{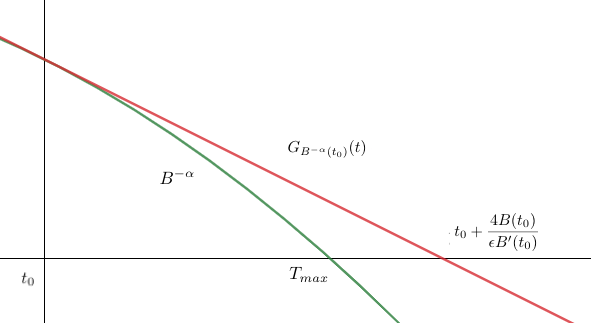}
\caption{Demonstration of the qualitative behavior that leads to the existence of $T_{max}$}
\label{demofig}
\end{figure}

\section{Applicable Cosmology Examples Revisited}

Here we will explore two specific cases for the expansion factor $a(t)$. The first is the de Sitter case, where $a(t)=e^{Ht}$ for some $H>0$. This case serves as a manner of ``upper limit'' on the rate of growth of $a(t)$, as we can see if we consider $H=\frac{\epsilon}{6n}$. We then see $\frac{\Dot{a}}{a} = \frac{\epsilon}{6n}$, the upper limit of this ratio from our general case. Clearly for any function consistently above exponential order this ratio would result in an increasing function, which could not be bounded forward in time as we require.

The other case is the FLRW expansion factor $a(t)=t^k$ such that $k>1$. Here we see the benefit of being well under exponential order, but we find our $\frac{\Dot{a}}{a}$ ratio is undefined at $t=0$. This, as we will explain, is the inspiration for our $t_0$ requirement in the general case of our argument above. While many expansion factors will allow for $t_0=0$ (Minkowski spacetime
and de Sitter come to mind), extreme behavior at $t=0$ can require $t_0$ to be placed adequately far in the future.

%In each of these cases we present the inequalities in the simplified m=1 case for clarity and brevity.

\subsection{The de Sitter Case}

Of course, in the de Sitter case, where we have $a(t)=e^{Ht}$, we 
satisfy our conditions on $a(t)$ as well as the energy inequality with 
$t_0=0$ with the restriction that $H\leq \frac{\epsilon}{6n}$. This 
gives us $\frac{\Dot{a}}{a}=H \leq \frac{\epsilon}{6n}$ and 
$E(t_0)\geq E(t) \forall t\in [0,\infty)$, clearly.

The argument for Lemma \ref{invariant}
therefore proceeds as follows in this case:
\begin{equation} \label{eq1dS}
\begin{split}
E(t_0)&\geq E(t)\\
&=\frac{1}{2}\| u_t\|^2 + \frac{1}{2}m^2\| u \|^2 + \frac{1}{2}e^{-2H}\|\nabla u \|^2 - \int_{\mathbb{R}^n}^{} F(u) d x  \\
&\geq \frac{1}{2}\| u_t\|^2 + \frac{1}{2}m^2\| u \|^2 + \frac{1}{2}e^{-2H}\|\nabla u \|^2 - \frac{1}{\epsilon + 2} \int_{\mathbb{R}^n} u f(u) d x\\
&=\frac{1}{2}\| u_t\|^2 + (\frac{1}{2} - \frac{1}{\epsilon+2})(m^2\|u\|^2 + e^{-2H}\|\nabla u\|^2) + \frac{1}{\epsilon + 2}I(u)\\
&\;\;\;\;-\frac{nH}{\epsilon+2}(\tilde{m}^2 u_0,u_1)
\end{split}
\end{equation}

And following along the same lines as the general argument, with our restriction on $H$, we arrive at the same contradiction arrived at with \ref{ineqend}.

This particular example is enlightening not only because of the 
pertinence to the current assumed nature of our universe, but also 
for making clear that $t_0$ is only a positive quantity in cases 
where our $a(t)$ conditions cannot be satisfied over all nonnegative 
$t$, whether due to discontinuity or some other 
early behavior associated with the expansion factor.

This example also matches the intuition that one would have for how 
blow-up would be influenced by our choice of $a$. The limit on $H$ 
serves as a limit on on the rate of expansion, which if too great 
might ``smooth out" too quickly to allow for blowup. Furthermore,
the local
existence result in Theorem 1.2 of \cite {Gal1} also requires that
 $H$ be small relative to $m$.
\subsection{The FLRW Case} \label{FLRWC}

Here we have $a(t)=t^k$ and the metric is therefore expressed as follows:
\begin{equation}
    {\mathbf{g}} := -dt^2 + t^{2k}(dx_1^2+dx_2^2+...+dx_n^2) .
\end{equation}

When we consider the $\frac{\Dot{a}}{a}$ ratio we arrive at:
\begin{equation}
    \frac{\Dot{a}(t)}{a(t)} = knt^{-1}
\end{equation}
Thus the discontinuity at $t=0$ requires that we consider our initial time $t_0$ as follows:
\begin{equation}\label{FLRW_initial_time}
	\frac{6kn}{\epsilon}= t_0 . 
\end{equation}

\begin{equation}
\begin{split}
E(t_0)&\geq E(t)\\
&=\frac{1}{2}\| u_t\|^2 + \frac{1}{2}m^2\| u \|^2 + \frac{1}{2}t^{-2k}\|\nabla u \|^2 - \int_{\mathbb{R}^n}^{} F(u) d x  \\
&\geq \frac{1}{2}\| u_t\|^2 + \frac{1}{2}m^2\| u \|^2 + \frac{1}{2}t^{-2k}\|\nabla u \|^2 - \frac{1}{\epsilon + 2} \int_{\mathbb{R}^n} u f(u) d x\\
&=\frac{1}{2}\| u_t\|^2 + (\frac{1}{2} - \frac{1}{\epsilon+2})(m^2 \|u\|^2 + t^{-2k}\|\nabla u\|^2) + \frac{1}{\epsilon + 2}I(u)\\
&\;\;\;\;-\frac{knt^{-1}}{\epsilon+2}(\tilde{m}^2u_0,u_1)
\end{split}
\end{equation}
Again, following the lines of the general argument, we quickly arrive at 
the same contradiction of \ref{innereq} that allows the proof to proceed.
With the curtailed expansion of this polynomial case, we see that there are 
no restrictions reminiscent of those that had to be placed on H in the de 
Sitter case. Rather, here we have to contend with the discontinuity at 
$t=0$, which is avoided by placing our start time suitably far from that 
initial singularity.

It is worth noting that the general argument is widely applicable,
and these simple, well-studied metrics fall within the restrictions
on the expanding factor required for the argument. One can also think
of the exponential de Sitter case as an upper bound for this factor,
as it is the most rapidly expanding expansion factor that can remain
within these restrictions.

\section{Acknoledgements}

Jes\'us Oliver was supported by an AMS-Simons Research Enhancement Grant for Primarily Undergraduate Institution Faculty.

\section{Appendix: Lower Bound for the Initial Data}

In this section we prove that the assumptions
of our main theorem imply that the norm of the
initial data
$u_0(x)\in H_0^1$, $u_1(x)\in L^2$
has a positive uniform lower bound. 

\begin{theorem}
    Suppose that all the hypotheses of Theorem \ref{main_result} are satisfied.
    Then there exists a uniform constant $C_{min}>0$ such that
    \[C_{min}<\|u_0\|_{H_0^{1}}\]
\end{theorem}

\begin{proof}
By the density of $C^{\infty}_{c}$ in $H^1_0$,
it suffices to consider smooth,
compactly supported functions. 
The condition $u_0\in \mathcal{B}$ (equivalently $I(u(t_0))<0$) implies
\begin{equation}
   \int_{\{t=t_0\}\times\mathbb{R}^n}\left(m^2u^2 + 
   \frac{1}{a^2}(\nabla u)^2 + n\frac{\Dot{a}}{a}\Tilde{m}^2uu_t
   \right) dx
   <  \int_{\{t=t_0\}\times\mathbb{R}^n} uf(u) \ d x.
   \label{app1}
\end{equation}
By assumption $m>0$; by Eq. \eqref{accel_exp} $a(t_0)$ and $\dot{a}(t_0)$ are also positive. By Eq. \eqref{preconditions} we also have a
lower bound for the last term on the LHS of \eqref{app1}.
Consequently
\begin{equation}
   0<\int_{\{t=t_0\}\times\mathbb{R}^n}\left(u^2 + 
   (\nabla u)^2 
   \right) dx+E(t_0)
   <  C\int_{\{t=t_0\}\times\mathbb{R}^n} uf(u) \ d x.
\end{equation}
Since $f$ is Lipshitz continuous with exponent $0\leq \alpha\leq \frac{2}{n-2}$, we have
\[|f(u)-f(0)|=|f(u)-0|\leq C|u|^{\alpha+1} . \]  
Using this together with the triangle inequality,
and $E(t_0)>0$
we can conclude
\begin{equation}
   \|u_0\|_{H_0^1}^2<C\int_{\{t=t_0\}\times\mathbb{R}^n} uf(u) \ d x
   <  C\|u_0\|_{L^{\alpha+2}}^{\alpha+2} \label{temp1}
\end{equation}
Next, using the Sobolev Embedding Theorem (see for example Thm. 27.18 in \cite{Driv})
we have the bound
\[\|u_0\|_{L^{q^{*}}}\leq C \|u_0\|_{H^{1}}\]
with $q^{*}=\frac{4}{n-2}$. Since
$0\leq \alpha\leq \frac{2}{n-2}$, we also have
\[\alpha\leq \frac{2}{n-2}<\frac{4}{n-2}=\frac{2n-2n+4}{n-2}
\qquad \Rightarrow \qquad
\alpha+2<\frac{2n}{n-2},
\]
Therefore \eqref{temp1} and the compact support of $u$ yields
\begin{equation*}
   \|u_0\|_{H_0^1}^2<  C\|u_0\|_{L^{\alpha+2}}^{\alpha+2}\leq
   C\|u_0\|_{L^{q^*}}^{q^*(\alpha+2)}< C_{min}\|u_0\|^{\alpha+2}_{H_0^{1}}
\end{equation*}
where $C_{min}$ is a uniform constant depending only on the Sobolev
embedding inequality and on the inequalities assumed in Theorem \ref{main_result}. As a result, we
have the lower bound
\[C_{min}^{-\frac{1}{\alpha}}<\|u_0\|_{H_0^{1}}\]
from which the theorem follows.
\end{proof}

\bibliographystyle{plain}

\begin{thebibliography}{99}

\bibitem{Driv}
Bruce Driver,
\textit{Analysis Tools With Applications},
    Lecture Notes (2003). \url{https://mathweb.ucsd.edu/~bdriver/231-02-03/Lecture_Notes/Sobolev%20Inequalities.pdf}

\bibitem{Gal1}
Anahit Galstian, Karen Yagdjian,
\textit{Global solutions for semilinear Klein–Gordon equations in FLRW spacetimes}.
	Nonlinear Analysis: Theory, Methods and Applications. (2015) Vol 113, 339-356
\bibitem{Levine} 
Howard A. Levine,
\textit{Instability and Nonexistence of Global Solutions to Nonlinear Wave Equations of the Form $Pu_{tt} = -Au + F(u)$}. 
Transactions of the American Mathematical Society 192 (1974) 1–21
\bibitem{Levine2}
Howard A. Levine,
\textit{Some Additional Remarks on the Nonexistence of Global Solutions to Nonlinear Wave Equations} SIAM J. Math. Anal. 5 (1974) 138-146

\bibitem{Oliver}
Jesús Oliver,
\textit{A Vector Field Method for Non-Trapping, Radiating Spacetimes}.
	Journal of Hyperbolic Differential Equations. 13. 10.1142/S021989161650020X. (2014) 
\bibitem{Oliver2}
Jesús Oliver, Jacob Sterbenz,
\textit{A Vector Field Method for Radiating Black Hole Spacetimes}.
	Anal. PDE 13 (1) 29 - 92 https://doi.org/10.2140/apde.2020.13.29 (2020)
\bibitem{Wang} 
Yanjin Wang,
\textit{A Sufficient Condition for Finite Time Blow Up of the Nonlinear Klein-Gordon Equations with Arbitrarily Positive Initial Energy}. 
	Proceedings of the American Mathematical Society 136, no. 10 http://www.jstor.org/stable/20535570 (2008) 3477–82 
\bibitem{Yangxu} 
Yanbing Yang, Runzhang Xu,
\textit{Finite Time Blowup for Nonlinear Klein-Gordon Equations with Arbitrarily Positive Initial Energy}. 
Applied Mathematics Letters 77 (2018) 21–26
\\	
\end{thebibliography}
\end{document}